\documentclass[11pt]{amsart}
\usepackage{amssymb}

\newtheorem{thm}{Theorem}[section]

\newtheorem{lem}[thm]{Lemma}
\newtheorem{prop}[thm]{Proposition}

\theoremstyle{remark}
\newtheorem{ntn}[thm]{Notation}
\newtheorem{rem}[thm]{Remark}
\newtheorem{dft}[thm]{Definition}

\begin{document}

\title[Fixed point sets of distinguished collections]{On fixed point sets of distinguished collections
for groups of parabolic characteristic}

\author{John Maginnis}
\author{Silvia Onofrei}
\address{Department of Mathematics, Kansas State University, 137 Cardwell
Hall, Manhattan, Kansas 66506, Email: maginnis@math.ksu.edu}
\address{Department of Mathematics, Kansas State University, 137 Cardwell
Hall, Manhattan, Kansas 66506, Email: onofrei@math.ksu.edu}
\date{Friday, 31 August 2007}

\maketitle
\begin{abstract}
We determine the nature of the fixed point sets of groups of order $p$, acting on complexes of distinguished $p$-subgroups (those $p$-subgroups containing $p$-central elements in their centers). The case when $G$ has parabolic characteristic $p$ is analyzed in detail.
\end{abstract}

\section{Introduction}
The subgroup complexes associated to suitably chosen collections of $p$-subgroups are relevant to the understanding of the $p$-local structure of the underlying group $G$, and also provide valuable tools for investigating the modular representation theory and the mod-$p$ cohomology of the group $G$.

\medskip
This paper continues the systematic study, started in \cite{mgo3}, of certain collections of $p$-subgroups, which we call distinguished. These are subcollections of the standard collections of $p$-subgroups and which consist of those $p$-subgroups which contain $p$-central elements in their centers.

\medskip
A group $G$ has parabolic characteristic $p$ if all the $p$-local subgroups which contain a Sylow $p$-subgroup of $G$ have characteristic $p$. We study the homotopy type of the fixed point sets of subgroups of order $p$ acting on the complex $\Delta$ of distinguished $p$-radical subgroups of $G$, for groups of parabolic characteristic $p$. Let $P=\langle x \rangle$ be such a $p$-subgroup. If $x$ is $p$-central then the fixed point set is contractible. If $x$ is not of central type, then the homotopy type of the corresponding fixed point set is determined by the group structure of $C_G(x)$.
If $C_G(x)$ has characteristic $p$, then $\Delta^P$ is again contractible.
If we assume that $C_G(x)$ does not have characteristic $p$, and that $\overline{C}=C_G(x)/O_p(C_G(x))$ has parabolic characteristic $p$, then the fixed point set $\Delta ^P$ is equivariantly homotopy equivalent to the complex of distinguished $p$-radical subgroups of $\overline{C}$.

\medskip
In Section $2$, notation is introduced and a few basic results are reviewed. In Section $3$, two varieties of collections of $p$-subgroups are defined and various homotopy properties are described; the fixed point sets of $p$-central elements acting on the corresponding complexes are shown to be contractible. In Section $4$, under certain hypotheses, the fixed point sets of subgroups of order $p$ of noncentral type are shown to be equivariantly homotopy equivalent to the complex for a quotient of the centralizer. In Section $5$ we consider a few examples where $G$ is a sporadic finite simple group, applications to modular representation theory are also given.\\

\section{Notation, terminology and standard results}

Throughout this paper $G$ is a finite group and $p$ a prime dividing its order.

\smallskip
A $p$-subgroup $R$ of $G$ is called $p$-{\it radical} if $R = O_p(N_G(R))$, where $O_p(H)$ is the largest normal
$p$-subgroup of $H$. Every $p$-subgroup $Q$ of $G$ is contained in a $p$-radical subgroup of $G$ uniquely
determined by $Q$ and $G$. This is called the {\it radical closure} of $Q$ in $G$ and it is the last term $R_Q$ of the
chain $P_{i+1} = O_p(N_G(P_i))$ starting with $Q = P_0$.  It is easy to see that $N_G(Q) \leq N_G(R_Q)$. A $p$-subgroup $R$ is called $p$-{\it centric} if $Z(R)$ is a Sylow $p$-subgroup of $C_G(R)$, in which case $C_G(R) = Z(R) \times H$, with $H$ a subgroup of order relatively prime to $p$.

\medskip
A {\it collection} $\mathcal{C}$ of $p$-subgroups of $G$ is a set of $p$-subgroups which is closed under conjugation; a collection is a $G$-poset under the inclusion relation with $G$ acting  by conjugation. The nerve $|\mathcal{C}|$ is the simplicial complex which has as simplices proper inclusion chains in $\mathcal{C}$; the correspondence $\mathcal{C} \rightarrow |\mathcal{C}|$ allows assignment of topological concepts to posets \cite[Sect.1]{qu78}. A collection $\mathcal{C}$ is contractible if $|\mathcal{C}|$ is contractible. A poset map is a $G$-homotopy equivalence if and only if the induced map on $H$-fixed points is a homotopy equivalence for all $H \leq G$; see \cite[1.3]{tw91}.

\begin{ntn}For $Q \leq G$ a subgroup let $\mathcal{C}^P = \lbrace Q \in \mathcal{C} \; | \; P \leq N_G(Q) \rbrace$
denote the subcollection of $\mathcal{C}$ fixed under the action of $P$.
Next $\mathcal{C}_{>P} = \lbrace Q \in \mathcal{C} \; | \; P<Q \rbrace$.
Similarly define $\mathcal{C}_{\geq P}$ and also $\mathcal{C}_{< P}$ and $\mathcal{C}_{\leq P}$.
We will also use the notation $\mathcal{C}_{>P}^{\leq H}$ for the set $\mathcal{C}_{>P}^{\leq H} = \lbrace Q
\in \mathcal{C} \; | \;P<Q \leq H \rbrace$.
\end{ntn}

\begin{thm}Let $G$ be a finite group and $\mathcal{C} \subseteq \mathcal{D}$ two collections of subgroups.
\begin{itemize}
\item[(1)]\cite[Prop. 1.7]{tw91} Assume that $\mathcal{D}_{>P}$ is $N_G(P)$-contractible for all $P \in \mathcal{D} \setminus \mathcal{C}$. Then the inclusion $\mathcal{C} \hookrightarrow \mathcal{D}$ is a $G$-homotopy equivalence.
\item[(2)]\cite[Thm.1]{tw91} Assume either that $\mathcal{C}_{\geq P}$ is $N_G(P)$-contractible for all $P \in \mathcal{D}$, or that $\mathcal{C}_{\leq P}$ is $N_G(P)$-contractible for all $P \in \mathcal{D}$. Then the inclusion $\mathcal{C} \hookrightarrow \mathcal{D}$ is a $G$-homotopy equivalence.
\item[(3)]\cite[2.2(3)]{gs} Suppose that $F$ is a $G$-equivariant poset endomorphism of $\mathcal{C}$ satisfying either $F \geq Id_{\mathcal{C}}$ or $F \leq Id_{\mathcal{C}}$. Then, for any collection $\mathcal{C}'$
containing the image of $F$,
the inclusions $F(\mathcal{C}) \subseteq \mathcal{C}' \subseteq \mathcal{C}$ are $G$-homotopy equivalences.
\item[(4)]\cite[in proof Lemma 2.7(1)]{gs} Let $\mathcal{C}$ be a collection of $p$-subgroups that is closed under passage to $p$-overgroups. Let $H$ be an arbitrary $p$-subgroup in $G$. Then the inclusion $\mathcal{C}_{\geq H} \hookrightarrow \mathcal{C}^H$ is a homotopy equivalence.
\end{itemize}
\end{thm}

In what follows $\mathcal{A}_p (G)$ will denote the Quillen collection of nontrivial elementary abelian $p$-subgroups, $\mathcal{S}_p (G)$ the Brown collection of nontrivial $p$-subgroups and $\mathcal{B}_p (G)$ the Bouc collection of nontrivial $p$-radical subgroups. The inclusions $\mathcal{A}_p(G) \subseteq \mathcal{S}_p(G)$ and $\mathcal{B}_p(G) \subseteq \mathcal{S}_p(G)$ are $G$-homotopy equivalences \cite[Thm.2]{tw91}.

\medskip
Let $\mathcal{C}e_p(G)$ denote the subcollection of $\mathcal{S}_p(G)$ consisting of nontrivial $p$-centric subgroups and let $\mathcal{B}^{\text{cen}}_p(G)=\mathcal{C}e_p(G) \cap \mathcal{B}_p(G)$ be the collection of nontrivial $p$-radical and $p$-centric subgroups. These two collections are not in general homotopy equivalent with $\mathcal{S}_p(G)$; however the inclusion map $\mathcal{B}^{\text{cen}}_p(G) \subseteq \mathcal{C}e_ p(G)$ is a $G$-homotopy equivalence; see \cite[Thm.1.1]{gs}.

\bigskip
\begin{dft}The group $G$ has {\it characteristic} $p$ if $C_G(O_p(G)) \leq O_p(G)$. If all $p$-local subgroups
of $G$ have characteristic $p$ then $G$ has {\it local characteristic} $p$.
\end{dft}

\begin{prop}Assume $G$ has characteristic $p$. Then the following hold:
\begin{enumerate}
\item[(1)]\cite[12.6]{gls2b} $G$ has local characteristic $p$.
\item[(2)]\cite[Lemma 1]{sol74} If $P$ is a $p$-subgroup of $G$ and $H$ a subgroup of $G$ with $PC_G(P) \leq H \leq N_G(P)$, then $H$ has characteristic $p$.
\end{enumerate}
\end{prop}

\begin{prop}
Let $G$ be a finite group.
\begin{enumerate}\item[(1)] Let $P \leq G$ be a $p$-subgroup. Then $C_G(P)$ has characteristic $p$ if and only if $N_G(P)$ has characteristic $p$.
\item[(2)]\cite{sol74} Let $Q$ be a $p$-subgroup of a finite group $G$, with $C_G(Q)$ of characteristic $p$. Let $P$ be a $p$-subgroup of $G$ containing $Q$. Then $C_G(P)$ has characteristic $p$.
\end{enumerate}
\end{prop}

\begin{dft}A {\it parabolic subgroup} of $G$ is defined to be a subgroup which contains a Sylow $p$-subgroup of $G$. The group $G$ has {\it parabolic characteristic} $p$ if all $p$-local, parabolic subgroups of $G$ have characteristic $p$.
\end{dft}

\begin{rem} Examples of groups of local characteristic $p$ are the groups of Lie type defined over fields of characteristic $p$, some of the sporadic groups (such as $M_{22}, M_{24}, Co_2$ for $p=2$, $M_{11}$, $McL$ for $p=3$, $McL$, $Ly$ for $p=5$), and any groups with self-centralizing Sylow $p$-subgroup of order $p$, such as Alt$(p)$.
\end{rem}

\begin{rem}
Any group of local characteristic $p$ has parabolic characteristic $p$. Some examples of sporadic groups of parabolic characteristic $p$ are: $M_{12}, J_2, Co_1$ for $p=2$, $M_{12}, J_3, Co_1$ for $p=3$, $J_2, Co_1, Co_2$ for $p=5$.\\
\end{rem}

\section{Distinguished collections of $p$-subgroups}

An element $x$ of order $p$ in $G$ is {\it $p$-central} if $x$ is in the center of a Sylow $p$-subgroup of $G$. Let $\Gamma _p(G)$ denote the family of $p$-central elements of $G$. For a $p$-subgroup $P$ of $G$ define:
$$\widehat{P} = \langle x \; | \; x \in \Omega _1 Z(P) \cap \Gamma_p(G) \rangle$$
Further, for $\mathcal{C}_p(G)$ a collection of $p$-subgroups of $G$ denote:
$$\widehat{\mathcal{C}}_p(G) = \lbrace P \in \mathcal{C}_p(G) \; | \; \widehat{P} \not= 1\; \rbrace$$
the collection of subgroups in $\mathcal{C}_p(G)$ which contain $p$-central elements in their centers.
We call $\widehat{\mathcal{C}}_p(G)$ the {\it distinguished}
$\mathcal{C}_p(G)$ {\it collection}. We shall refer to the subgroups in
$\widehat{\mathcal{C}}_p(G)$ as {\it distinguished subgroups}. Also, denote
$$\widetilde{\mathcal{C}}_p(G) = \lbrace P \in \mathcal{C}_p(G) \; | \; P \cap \Gamma_p(G) \not= 1\; \rbrace$$
the collection of subgroups in $\mathcal{C}_p(G)$ which contain a $p$-central element.
Obviously $\widehat{\mathcal{C}}_p(G) \subseteq \widetilde{\mathcal{C}}_p(G) \subseteq \mathcal{C}_p(G)$.

\begin{prop} All $p$-centric subgroups of $G$ are distinguished,
that is $\mathcal{C}e_p(G) \subseteq \widehat{\mathcal{S}}_p(G)$. Consequently the collection of distinguished $p$-radical subgroups
contains the collection of $p$-centric and $p$-radical subgroups:
$\mathcal{B}_p^{\text{cen}}(G) \subseteq \widehat{\mathcal{B}}_p(G)$.
\end{prop}

\begin{proof} Let $P$ be a centric $p$-subgroup of $G$ and let $S$ be any Sylow
$p$-subgroup of $G$ which contains $P$. Then $Z(S) \leq C_G(P)$ and $Z(S) \leq Z(P)$.
\end{proof}

\begin{rem} If $R$ is a $p$-subgroup of $G$ and if $N_G(R)$
contains a Sylow $p$-subgroup of $G$, then $R$ is distinguished. This is
easy to see since in this case $R \triangleleft S \in
\text{Syl}_p(N_G(R)) \subseteq \text{Syl}_p(G)$. Thus $R \cap Z(S) \not= 1$.
\end{rem}

\begin{lem} Let $P \in \mathcal{S}_p(G)$ and assume that $N_G(P)$ has
characteristic $p$. Then $C_G(O_p(N_G(P)))=Z(O_p(N_G(P)))$, and thus $O_p(N_G(P))$
is $p$-centric and distinguished.
\end{lem}

\begin{proof} Since $P \leq O_p(N_G(P))$, we have $C_G(O_p(N_G(P))) \leq
C_G(P) \leq N_G(P)$.  Thus $C_G(O_p(N_G(P)))=C_{N_G(P)}(O_p(N_G(P))) \leq
O_p(N_G(P))$, and so $C_G(O_p(N_G(P)))=Z(O_p(N_G(P)))$.  Clearly $O_p(N_G(P))$ is
$p$-centric, and Proposition 3.1 implies this group is distinguished.
\end{proof}

\begin{prop} Let $G$ have local characteristic $p$. Then $\mathcal{B}_p^{\text{cen}}(G)=\mathcal{B}_p(G)$.
\end{prop}

\begin{proof} Let $R \in \mathcal{B}_p(G)$, so that $R=O_p(N_G(R))$.
Since $N_G(R)$ has characteristic $p$, $C_G(R)=Z(R)$ and $R \in \mathcal{B}_p^{\text{cen}}(G)$.
\end{proof}

\begin{prop} Let $G$ have parabolic characteristic $p$. Then:
\begin{itemize}
\item[(a)] If $P \in \widetilde{\mathcal{S}}_p(G)$ then $N_G(P)$ has characteristic $p$.
\item[(b)] $\mathcal{B}^{\text{cen}}_p(G) = \widehat{\mathcal{B}}_p(G) = \widetilde{\mathcal{B}}_p(G)$.
\end{itemize}
\end{prop}

\begin{proof} (a) Let $z \in P$ be a $p$-central element, so that $N_G(\langle z \rangle)$
contains a Sylow $p$-subgroup of $G$.  Thus $N_G(\langle z \rangle)$ has characteristic $p$.
By Proposition $2.5(1)$, $C_G(z)$ has characteristic $p$. By Proposition $2.5(2)$, $C_G(P)$ has
characteristic $p$, and another application of Proposition $2.5(1)$ shows $N_G(P)$ has characteristic $p$.

\smallskip
(b) Note that $\mathcal{B}^{\text{cen}}_p(G) \subseteq \widehat{\mathcal{B}}_p(G) \subseteq \widetilde{\mathcal{B}}_p(G)$. Let $R \in \widetilde{\mathcal{B}}_p(G)$. Then $N_G(R)$ has
characteristic $p$, which implies that $C_G(R) = Z(R)$. Thus $R \in \mathcal{B}^{\text{cen}}_p(G)$.
\end{proof}

\begin{rem} It follows from the above Proposition that if $G$ has parabolic characteristic $p$ and $V \in \mathcal{B}_p(G) \setminus \widehat{\mathcal{B}}_p(G)$, then $V$ does not contain any $p$-central elements.
\end{rem}

\begin{prop} If $G$ has parabolic characteristic $p$, then the collections $\widehat{\mathcal{B}}_p(G), \widehat{\mathcal{A}}_p(G)$ and $\widehat{\mathcal{S}}_p(G)$ are $G$-homotopy equivalent.
\end{prop}

\begin{proof}
We first show that the inclusion map $\widehat{\mathcal{A}}_p(G) \hookrightarrow \widehat{\mathcal{S}}_p(G)$ is a $G$-homotopy equivalence. We attain this result by showing that $\widehat{\mathcal{A}}_p(G)_{\leq P}$ is $N_G(P)$-contractible for any $P \in \widehat{\mathcal{S}}_p(G)$ and then applying Theorem $2.2(2)$. If $Q \in \widehat{\mathcal{A}}_p(G)_{\leq P}$ then $Q \widehat{P} \in \widehat{\mathcal{A}}_p(G)_{\leq P}$. We obtain the following contracting homotopy:
$Q \leq Q \widehat{P} \geq \widehat{P}$. The $N_G(P)$-contractibility follows from the fact that the two inequalities correspond to poset maps which are $N_G(P)$-equivariant.

\medskip
To show that $\widehat{B}_p(G)$ is $G$-homotopy equivalent to $\widehat{S}_p(G)$, we will use Theorem $2.2(1)$. Thus we have to prove that for each $P \in \widehat{\mathcal{S}}_p(G) \setminus
\widehat{\mathcal{B}}_p(G)$, the subcollection $\widehat{\mathcal{S}}_p(G)_{>P}$ is $N_G(P)$-contractible.

\medskip
Let $Q \in \widehat{\mathcal{S}}_p(G)_{>P}$, so $P< N_Q(P) \leq N_G(P)$.
The subgroup $N_Q(P)$ is distinguished since it contains $Z(Q)$.
As $P$ is a distinguished $p$-subgroup, $N_G(P)$ has characteristic $p$. Denote $O_{NP} = O_p(N_G(P))$ and observe that $P < O_{NP}$. By Lemma 3.3, $O_{NP}$ is $p$-centric and distinguished.  Since the collection of $p$-centric
subgroups is closed under passage to $p$-overgroups, $N_Q(P) O_{NP}$ is also $p$-centric and
thus it is distinguished.
Now consider the string of poset maps $\widehat{\mathcal{S}}_p(G)_{>P} \rightarrow
\widehat{\mathcal{S}}_p(G)_{>P}$ given by:
$$Q \geq N_Q(P) \leq N_Q(P) O_{NP} \geq O_{NP}$$
which proves the $N_G(P)$-contractibility of $\widehat{\mathcal{S}}_p(G)_{>P}$.
\end{proof}

\begin{prop}Assume that $G$ has parabolic characteristic $p$; then the inclusion $\widehat{\mathcal{S}}_p(G) \hookrightarrow \widetilde{\mathcal{S}}_p(G)$ is a $G$-homotopy equivalence.
\end{prop}

\begin{proof} We will show that the subcollection $\widehat{\mathcal{S}}_p(G)_{\geq P}$ is $N_G(P)$-contractible for every $P \in \widetilde{\mathcal{S}}_P(G)$ and then apply Theorem $2.2(2)$. Let $Q \in \widehat{\mathcal{S}}_p(G)_{\geq P}$ so $P \leq N_Q(P) \leq Q$ and $N_Q(P)$ is a distinguished $p$-subgroup with $Z(Q) \leq Z(N_Q(P))$. Next consider $R_P$, the radical closure of $P$ defined in Section $2$. Note that $N_Q(P)R_P$ is a subgroup in $N_G(R_P)$ since $N_Q(P) \leq N_G(P) \leq N_G(R_P)$. By Proposition $3.5$, $N_G(R_P)$ has characteristic $p$ and $R_P$ is $p$-centric.
Thus $N_Q(P) R_P$ is also $p$-centric, and both $R_P$ and $N_Q(P) R_P$ are distinguished.
Now consider the string of $N_G(P)$-equivariant poset maps $\widehat{\mathcal{S}}_p(G)_{\geq P} \rightarrow \widehat{\mathcal{S}}_p(G)_{\geq P}$ given by:
$$Q \geq N_Q(P) \leq N_Q(P) R_P \geq R_P$$
which proves the contractibility of the subcollection $\widehat{\mathcal{S}}_p(G)_{\geq P}$.
\end{proof}

\begin{prop} The fixed point set $\widetilde{\mathcal{S}}_p^Z(G)$ is $N_G(Z)$-contractible whenever $Z=\langle z \rangle$ with $z$ a $p$-central element in $G$.
\end{prop}

\begin{proof} First note that $Z \in \widetilde{\mathcal{S}}_p^Z(G)$. If $P \in \widetilde{\mathcal{S}}_p^Z(G)$ then $PZ \in \widetilde{\mathcal{S}}_p^Z(G)$ since $Z$ normalizes $PZ$. There is a contracting homotopy $P \leq PZ \geq Z$ via $N_G(Z)$-equivariant maps.
\end{proof}

\begin{rem} Since $\widehat{\mathcal{S}}_p(G), \widehat{\mathcal{A}}_p(G), \widehat{\mathcal{B}}_p(G)$ and $\widetilde{\mathcal{S}}_p(G)$ are $G$-homotopy equivalent, for $G$ of parabolic characteristic $p$, $\widehat{\mathcal{S}}_p^Z(G)$, $\widehat{\mathcal{A}}_p^Z(G)$ and $\widehat{\mathcal{B}}_p^Z(G)$ are $N_G(Z)$-contractible as well.\\
\end{rem}

\section{Fixed point sets for noncentral elements}

We shall investigate the fixed point set of an element of order $p$ of noncentral type; these are elements of order $p$ in $G$ which are not conjugate to any element in the center of a Sylow $p$-subgroup of $G$. Under certain hypotheses, we will prove that the fixed point set is equivariantly homotopy equivalent to the complex for a quotient of the centralizer. This will require a combination of nine homotopy equivalences.

\begin{ntn}
Throughout this section, $T$ will be a subgroup of order $p$ of noncentral type in $G$. We will use the
shorthand notation $C=C_G(T)$ and $O_C=O_p(C)$.  The quotient group will be denoted
$\overline{C} = C/O_C$; the quotient map is $q : C \rightarrow \overline{C}$.  For $H \leq C$,
let $\overline{H} = q(H)$.  For $Q \leq C$, denote $O_Q = O_p(N_C(Q))$;
for $\overline{Q} \leq \overline{C}$, denote $O_{\overline{Q}} = O_p(N_{\overline{C}}(\overline{Q}))$.
Let $S_T \in Syl_p(C)$, and extend it to $S \in Syl_p(G)$.  Since $T \leq S_T \leq S$, we have $Z(S) \leq C$;
thus $Z(S) \leq S_T$ and in fact $Z(S) \leq Z(S_T)$. Note that $\overline{S_T} = q(S_T) \in Syl_p(\overline{C})$.
\end{ntn}

\begin{rem}
The proof of our main result (Theorem $4.12$) will require the following hypotheses:
\begin{enumerate}
\item $G$ is a finite group of parabolic characteristic $p$;
\item $C = C_G(T)$
does not have characteristic $p$;
\item The quotient group $\overline{C} = C_G(T)/O_p(C_G(T))$ has parabolic characteristic $p$.
\end{enumerate}
\end{rem}

We first recall a result which is due to Grodal \cite[pp. 420-421]{gr02}, see also
Sawabe \cite[Thm.1]{sa06}. For completeness we provide a proof.

\begin{prop}(in the proof of \cite[Thm.1.1]{gr02}) Let $\mathcal{C}$ be a collection of nontrivial $p$-subgroups of $G$, which is closed under passage to $p$-overgroups. Let $Q \in \mathcal{S}_p(G)$. If $\mathcal{C}'$ is a collection satisfying: $\mathcal{C} \cap \mathcal{B}_p(G) \subseteq \mathcal{C}' \subseteq \mathcal{C}$ then $\mathcal{C}_{>Q}$ is $N_G(Q)$-homotopy equivalent to $\mathcal{C}'_{>Q}$.
\end{prop}

\begin{proof}To simplify the notation, we shall denote by $\mathcal{X} = \mathcal{C}_{>Q}$ and by $\mathcal{Y} = \mathcal{C}'_{>Q}$. We will prove that $\mathcal{X}_{>P}$ is $N_G(P)$-contractible for all $P \in \mathcal{X} \setminus \mathcal{Y}$. Then, an application of Theorem $2.2(1)$ will give the result.
Note that $\mathcal{X}_{>P}= \lbrace R \in \mathcal{C} \; | \; P<R \rbrace $.
Denote $O_{NP}=O_p(N_G(P))$; since $P$ is not $p$-radical, $P < O_{NP}$.
By elementary group theory, for $P<R$, we have $P < N_R(P) \leq R$.
Since $\mathcal{C}$ is closed under passage to $p$-overgroups, the subgroups $N_R(P), N_R(P)O_{NP}$ and $O_{NP}$
are also in $\mathcal{X}_{>P}$. Thus we obtain a contracting homotopy given by the string of
$N_G(P)$-equivariant poset maps $\mathcal{X}_{>P} \rightarrow \mathcal{X}_{>P}$ :
$$R \geq N_R(P) \leq N_R(P) O_{NP} \geq O_{NP}.$$
Hence $\mathcal{X}_{>P}$ is $N_G(P)$-contractible.
\end{proof}

\begin{prop} Let $G$ be a finite group of parabolic characteristic $p$.
The inclusion $\widehat{\mathcal{S}}_p(G)^{\leq C}_{>T} \hookrightarrow
\widehat{S}_p^T(G)$ is a $N_G(T)$-homotopy equivalence.
\end{prop}

\begin{proof} We verify the following chain of $N_G(T)$-homotopy equivalences:
$$\widehat{\mathcal{S}}_p^T(G) \simeq \widetilde{\mathcal{S}}_p^T(G)  \simeq \widetilde{\mathcal{S}}_p(G)_{\geq T}
= \widetilde{\mathcal{S}}_p(G)_{>T} \simeq \widehat{\mathcal{S}}_p(G)_{>T} \simeq
\widehat{\mathcal{S}}_p(G)^{\leq N_G(T)}_{>T} = \widehat{\mathcal{S}}_p(G)^{\leq C}_{>T}.$$

The subcollections $\widehat{\mathcal{S}}_p^T(G)$ and $\widetilde{\mathcal{S}}_p^T(G)$ are $N_G(T)$-homotopy equivalent because $\widetilde{\mathcal{S}}_p(G)$ and $\widehat{\mathcal{S}}_p(G)$ are $G$-homotopy equivalent. The next step follows by an application of Theorem $2.2(4)$. The collection $\widetilde{\mathcal{S}}_p(G)$ is closed under passage to $p$-overgroups; thus the inclusion $\widetilde{\mathcal{S}}_p(G)_{\geq T} \hookrightarrow \widetilde{\mathcal{S}}_p^T(G)$ is a homotopy equivalence. Next, since $T \not \in \widetilde{\mathcal{S}}_p(G)$ it follows that $\widetilde{\mathcal{S}}_p(G)_{\geq T} = \widetilde{\mathcal{S}}_p(G)_{>T}$. The homotopy equivalence between $\widetilde{\mathcal{S}}_p(G)_{>T}$ and $\widehat{\mathcal{S}}_p(G)_{>T}$ follows from an application of Proposition $4.3$ with $\mathcal{C} = \widetilde{\mathcal{S}}_p(G)$ and $\mathcal{C}'=\widehat{\mathcal{S}}_p(G)$.
Note that Proposition $3.5(b)$ provides the necessary hypothesis
$\mathcal{C} \cap \mathcal{B}_p(G) \subseteq \mathcal{C}'$.
To see that $\widehat{\mathcal{S}}_p(G)^{\leq N_G(T)}_{>T} \hookrightarrow \widehat{\mathcal{S}}_p(G)_{>T}$ is a homotopy equivalence, consider the poset map $\widehat{\mathcal{S}}_p(G)_{>T} \rightarrow \widehat{\mathcal{S}}_p(G)_{>T}$ given by $Q \mapsto N_Q(T)$ whose image lies in $\widehat{\mathcal{S}}_p(G)^{\leq N_G(T)}_{>T}$, and apply
Theorem $2.2(3)$. The final equality follows from $T \leq P \leq N_G(T)$ if and only if $T \leq P \leq C$, since
$T \trianglelefteq P$ implies that $T \leq Z(P)$.
\end{proof}

\begin{lem}Let $G$ be a finite group of parabolic characteristic $p$.
Then $O_C \in \widetilde{\mathcal{S}}_p(G)$ if and only if $C$ has characteristic $p$.
\end{lem}

\begin{proof}
If $O_C \in \widetilde{\mathcal{S}}_p(G)$ then $N_G(O_C)$ has characteristic $p$, by Proposition $3.5(a)$.
As $T \leq O_C \trianglelefteq C \leq N_G(O_C)$, it follows that $C_{N_G(O_C)}(T)=C_G(T)=C$ and so
$C= T \cdot C_{N_G(O_C)}(T)$. Thus $C$ has characteristic $p$, by Proposition $2.4(2)$. Conversely,
assume that $C$ has characteristic $p$. Note $T \leq O_C$, so $C_G(O_C) \leq C$. Thus $C_G(O_C) =
C_C(O_C) \leq O_C$, and so $O_C$ is $p$-centric. Thus $O_C \in \widehat{\mathcal{S}}_p(G) \subseteq
\widetilde{\mathcal{S}}_p(G)$.
\end{proof}

\begin{prop}
Let $G$ be a finite group of parabolic characteristic $p$, and assume that
$O_C \in \widetilde{\mathcal{S}}_p(G)$. Then the fixed point set $\widehat{\mathcal{S}}_p(G)^T$ is contractible.
\end{prop}

\begin{proof}
Consider the poset map $\varphi : \widetilde {\mathcal{S}}_p(G)^{\leq C}_{>T}
\rightarrow \widetilde {\mathcal{S}}_p(G)^{\leq C}_{>T}$ given by $\varphi (P)=P \cdot O_C$. If $O_C$ contains $p$-central elements, the poset map $\varphi$
will have image equal to $\widetilde{S}_p(G)^{\leq C}_{\geq O_C}$; this is contractible, a cone on $O_C$. Apply Theorem $2.2(3)$; combining this with Proposition $4.4$ proves the result.
\end{proof}

\begin{prop} Let $G$ be a finite group of parabolic characteristic $p$, and assume that $C$
does not have characteristic $p$.
The inclusion $\mathcal{X} = \widetilde {\mathcal{S}}_p(G)^{\leq C}_{>O_C} \hookrightarrow
\mathcal{Y} = \widetilde {\mathcal{S}}_p(G)^{\leq C}_{>T}$ is a $N_G(T)$-homotopy equivalence.
\end{prop}

\begin{proof} The poset map $\varphi : \widetilde {\mathcal{S}}_p(G)^{\leq C}_{>T}
\rightarrow \widetilde {\mathcal{S}}_p(G)^{\leq C}_{>T}$ given by $\varphi (P)=P \cdot O_C$
now has image in $\widetilde {\mathcal{S}}_p(G)^{\leq C}_{>O_C}$ since $O_C$ is
purely noncentral, containing no $p$-central elements of $G$.
The result follows by an application of Theorem $2.2(3)$.
\end{proof}

\begin{prop} Let $G$ be a finite group of parabolic characteristic $p$.
The inclusion $\mathcal{X} = \widehat {\mathcal{S}}_p(G)^{\leq C}_{>H}
\hookrightarrow \mathcal{Y} = \widetilde {\mathcal{S}}_p(G)^{\leq C}_{>H}$ is a $N_G(T)$-homotopy
equivalence, where $H$ satisfies $T \leq H \leq C$ and $N_G(T) \leq N_G(H)$.
\end{prop}

\begin{proof}
Let $P \in \mathcal{Y}$; we will show that $\mathcal{X}_{\geq P}$ is equivariantly contractible and apply Theorem $2.2(2)$. Let $R_P$ be the radical closure of $P$; note that $T < P \leq R_P$, so that
$Z(R_P) \leq C$.  By Proposition $3.5(b)$, $R_P$ is $p$-centric and distinguished.  This implies that $Z(R_P)$ is distinguished, and also $P \cdot Z(R_P)$ is distinguished, since $Z(R_P) \leq Z(P \cdot Z(R_P))$. Let $Q \in \mathcal{X}_{\geq P}$. $P \leq Q$ implies $Z(Q) \leq Z(N_Q(P))$, so that $N_Q(P)$ is distinguished. Observe that $N_Q(P) \leq N_G(P) \leq N_G(R_P)$, implying $N_Q(P) Z(R_P)$ is a group.
Choose $S_R \in Syl_p(N_G(R_P))$ satisfying $N_Q(P) \leq S_R$, and extend to
$S \in Syl_p(G)$. Note $R_P \leq S_R \leq S$. Then $Z(S) \leq C_G(N_Q(P))$ and $Z(S) \leq C_G(R_P)$, so that $Z(S) \leq Z(R_P)$. This implies that $Z(S) \leq Z(N_Q(P)Z(R_P))$, so that $N_Q(P)Z(R_P)$ is distinguished.
Now consider the string of equivariant poset maps $\mathcal{X}_{\geq P} \rightarrow
\mathcal{X}_{\geq P}$ given by:
$$Q \geq N_Q(P) \leq N_Q(P) Z(R_P) \geq P \cdot Z(R_P)$$
which proves the equivariant contractibility of the subcollection $\mathcal{X}_{\geq P}$.
\end{proof}

\bigskip
There is no obvious relationship among those elements which are $p$-central in $G$, or in $C$, or in $\overline{C}$. In order to overcome this difficulty, we define a subcollection $\frak{S}$ of $\widehat{\mathcal{S}}_p(G)^{\leq C}_{>O_C}$ as follows:

\begin{dft}
Assume the notation from $4.1$ and set:
$$\frak{S} = \lbrace P \in \widehat{\mathcal{S}}_p(G)^{\leq C} _{>O_C} \; \Big| \;
Z(P) \cap Z(S) \not = 1, \; \text{for \; some} \; S_T \; {\rm and} \; S \; {\rm with} \;
P \leq S_T \leq S \rbrace$$
\end{dft}

This subcollection is also contained in $\widehat{\mathcal{S}}_p(C)$,
and it contains all $p$-centric subgroups $P$ in $C$ which properly contain $O_C$, since if $P \leq S_T \leq S$, then $Z(S) \leq C_G(P)$, implying $Z(S) \leq Z(P)$.

\bigskip
\begin{prop} Let $G$ be a finite group of parabolic characteristic $p$.
The inclusion $\frak{S} \hookrightarrow \widehat{\mathcal{S}}_p(G)^{\leq C} _{>O_C}$
is an $N_G(T)$-homotopy equivalence.
\end{prop}

\begin{proof}We will apply Theorem $2.2(2)$ once again. We need to show that $\frak{S}_{\geq Q}$ is
equivariantly contractible whenever $Q \in \widehat{\mathcal{S}}_p(G)^{\leq C} _{>O_C}$. Since $Q$ is a
distinguished $p$-subgroup of $G$, $N_G(Q)$ has characteristic $p$, by Proposition $3.5(a)$. Also
$QC_G(Q) \leq N_{C}(Q) \leq N_G(Q)$ since $T < Q \leq C$; by Proposition $2.4(2)$, it follows that
$N_{C}(Q)$ has characteristic $p$. Set $O_Q = O_p(N_{C}(Q))$ and observe that, according to Lemma $3.3$, $O_Q$ is $p$-centric. For $P \leq C$, $N_P(Q)O_Q$ is a group which is also $p$-centric.  Thus $O_Q$ and $N_P(Q)O_Q$ lie in $\frak{S}$. Assume now that $P \in \frak{S}_{\geq Q}$ and
consider the contracting homotopy given by the following string of equivariant poset maps:
$$P \geq N_P(Q) \leq N_P(Q) O_Q \geq O_Q.$$
This concludes the proof of the Proposition.
\end{proof}

\begin{prop} Let $G$ be a finite group of parabolic characteristic $p$, and assume that $C$ does not have characteristic $p$. Also assume that $\overline{C}$ has parabolic characteristic $p$. Then the map $q_*: \frak{S} \rightarrow \widehat{\mathcal{S}}_p(\overline{C})$ induced by the quotient map $q: C \rightarrow \overline{C}$ is a homotopy equivalence.
\end{prop}

\begin{proof}
To see that $q_*(\frak{S}) \subseteq \widehat{\mathcal{S}}_p(\overline {C})$, let $P \in \frak{S}$. Recall that $Z(S) \leq Z(S_T)$ and $\overline{S_T} \in Syl_p(\overline{C})$. Thus $\overline{Z(S)} \leq \overline{Z(S_T)} \leq Z( \overline{S_T} )$. Since $O_C$ is purely noncentral (Lemma $4.5$), the map $q : C \rightarrow \overline{C}$ is injective on the $p$-central elements of $Z(S)$. Therefore $Z(S) \cap Z(P) \not = 1$ implies $Z(\overline{S_T}) \cap Z(\overline{P})\not = 1$, and we have $q_*(P)= \overline{P} \in \widehat{\mathcal{S}}_p(\overline {C})$.

\medskip
According to a result of Th{\'e}venaz and Webb \cite[Thm.1]{tw91}, the poset map
$q_*: \frak{S} \rightarrow \widehat{\mathcal{S}}_p(\overline{C})$ is an equivariant homotopy equivalence if
$q_*^{-1}(\widehat{\mathcal{S}}_p(\overline{C})_{\geq {\overline{Q}}})$ is equivariantly contractible for any
$\overline{Q} \in \widehat{\mathcal{S}}_p(\overline{C})$. Define $Q = q^{-1}(\overline{Q})$.
Then $q_*^{-1}(\widehat{\mathcal{S}}_p(\overline{C})_{\geq \overline{Q}}) =
\lbrace P \in \frak{S} | \overline{Q} \leq \overline{P} \rbrace =
\lbrace P \in \frak{S} | Q \leq P \rbrace$, since $O_C \leq P$. This latter set is just
$\frak{S}_{\geq Q}$.
Given $P \in \frak{S}_{\geq Q}$, consider the string of equivariant poset maps
$\frak{S}_{\geq Q} \rightarrow \frak{S}_{\geq Q}$ given by:
$$P \geq N_P(Q) \leq N_P(Q) O_Q \geq O_Q.$$

We need to show that all of these terms lie in $\frak{S}_{\geq Q}$.  Note that
$N_P(Q) \leq P \leq S_T \leq S$, and $Q \leq P$ implies that $Z(P) \leq Z(N_P(Q))$.
Thus $N_P(Q) \in \frak{S}$. Next, we have $N_C(Q) = q^{-1}(N_{\overline{C}}(\overline{Q}))$,
using $O_C \leq Q$.  Thus $O_Q := O_p(N_C(Q))$ is equal to $q^{-1}(O_{\overline{Q}})$ by the
correspondence theorem for normal subgroups applied to $N_C(Q) \rightarrow N_{\overline{C}}
(\overline{Q})$. Since $\overline{Q} \in \widehat{\mathcal{S}}_p(\overline {C})$,
$N_{\overline{C}}(\overline{Q})$ has characteristic $p$, by Proposition $3.5(a)$.
Then $C_{\overline{C}}(O_{\overline{Q}}) \leq O_{\overline{Q}}$. Since $T \leq Q \leq O_Q$,
$C_G(O_Q) \leq C_G(T)=C$ and so $C_G(O_Q)=C_C(O_Q)$. Therefore $C_G(O_Q) = C_C(O_Q) \leq
q^{-1}(C_{\overline{C}}(O_{\overline{Q}})) \leq q^{-1}(O_{\overline{Q}}) = O_Q$. The
group $O_Q$ is $p$-centric in $G$. But $\frak{S}$ contains all subgroups of $C$, properly containing
$O_C$, which are $p$-centric in $G$. The $p$-overgroup $N_P(Q)O_Q$ is also $p$-centric
and lies in $\frak{S}$.
\end{proof}

\medskip
We can now state the main result of this section.

\begin{thm} Maintain the notation in $4.1$ and the hypotheses in Remark $4.2$.
There is an $N_G(T)$-equivariant homotopy equivalence
$\widehat{\mathcal{S}}^T_p(G) \simeq  \widehat{\mathcal{S}}_p(\overline{C})$.
\end{thm}

\begin{proof}
We have the chain of $N_G(T)$-homotopy equivalences:
$$\widehat{\mathcal{S}}^T_p(G) \simeq \widehat{\mathcal{S}}_p(G)^{\leq C}_{>T} \simeq
\widetilde{\mathcal{S}}_p(G)^{\leq C}_{>T} \simeq \widetilde{\mathcal{S}}_p(G)^{\leq C}_{>O_C} \simeq
\widehat{\mathcal{S}}_p(G)^{\leq C}_{>O_C} \simeq \frak{S} \simeq \widehat{\mathcal{S}}_p(\overline{C}).$$
The first step is Proposition $4.4$; then apply Proposition $4.8$ with $H = T$. Next, use
Proposition $4.7$, and then Proposition $4.8$ again with $H=O_C$. Finally, a combination of
Propositions $4.10$ and $4.11$ completes the proof of the theorem.
\end{proof}

\medskip
\section{Examples and Lefschetz modules}

We will discuss three examples, and give an application to modular representation theory.
Recall that if a group $G$ acts on a simplicial complex $\Delta$, we can construct the
virtual Lefschetz module by taking the alternating sum of the vector spaces (over a field
of characteristic $p$) spanned by the chains. To obtain the reduced Lefschetz module,
subtract the trivial one dimensional representation. Information about fixed point sets
leads to details about the vertices of irreducible summands of this module.\\

{\it The Fischer group $Fi_{22}$ and $p=2$}\\
We begin with the sporadic simple Fischer group $Fi_{22}$,
which has parabolic characteristic $2$ and has three conjugacy classes of involutions,
denoted $2A, 2B$ and $2C$ in the Atlas \cite{atlas}. The class $2B$ is 2-central. Their centralizers are $C_{Fi_{22}}(2A)=2.U_6(2)$,
\ $C_{Fi_{22}}(2B)=(2 \times 2^{1+8}_+:U_4(2)):2$ and $C_{Fi_{22}}(2C)=2^{5+8}:(S_3 \times 3^2:4)$.

\medskip
We consider the simplicial complex $\Delta$ whose vertex stabilizers are the four 2-local subgroups of $Fi_{22}$.
\begin{align*}
H_1 & = (2 \times 2^{1+8}_+ : U_4(2)) : 2 && H_2  = 2^{5+8}:(S_3 \times A_6)\\
H_3 & = 2^6:Sp_6(2) && H_4 = 2^{10}:M_{22}
\end{align*}

The flag stabilizers are listed below.
\begin{align*}
H_{1,2}&=2^{5+8}:(S_3 \times S_4) &H_{1,2,3}&=2^6:2^5:(2 \times S_4)\\
H_{1,3}&=2^6:2^5:S_6 & H_{1,2,4}&=2^{5+8}:(2 \times S_4)\\
H_{1,4}&=2^{10}:2^4:S_5 & H_{1,3,4}&=2^6:2^5:(2 \times S_4)\\
H_{2,3}&=2^6:2^{1+6}:(S_3 \times S_3) &  H_{2,3,4}&=2^6:2^{1+6}:(S_3 \times 2)\\
H_{2,4}&=2^{5+8}:(2 \times A_6) & H_{1,2,3,4}&=2^6:2^{1+6}:2^2\\
H_{3,4}&=2^6:2^6:L_3(2)
\end{align*}

The subgroup complex $\Delta$ is also known as the ``standard" $2$-local geometry for $Fi_{22}$. The geometry $\Delta$ is $G$-homotopy equivalent to $\mathcal{B}_2^{\rm cen}(Fi_{22})$, and since $Fi_{22}$ has parabolic characteristic $2$, this is equal to the distinguished collection $\widehat{\mathcal{B}}_2(Fi_{22})$; for details we refer the reader to Benson and Smith \cite[Sections 8.16 and 9.4]{bs04}.

\medskip
We shall use the notation from the Modular Atlas homepage \cite{moc}, where $\varphi _i$ denotes an irreducible module
of $Fi_{22}$ and $P_{Fi_{22}}(\varphi_i)$ is its corresponding projective cover.

\begin{prop}
Let $\Delta$ be the standard 2-local geometry for the Fischer group $Fi_{22}$.

\item[(a)] The reduced Lefschetz module is
$$\widetilde{L}_{Fi_{22}}(\Delta) = -P_{Fi_{22}}(\varphi_{12})-P_{Fi_{22}}(\varphi_{13}) - 6 \varphi_{15} -12 P_{Fi_{22}}(\varphi_{16}) -\varphi_{16}.$$
\item[(b)] The fixed point sets $\Delta^{2B}$ and $\Delta^{2C}$ are contractible.
\item[(c)] The fixed point set $\Delta^{2A}$ is equivariantly homotopy equivalent to
the building for the Lie group $U_6(2)$.
\item[(d)] There is precisely one nonprojective summand of the reduced Lefschetz module, it has vertex $\langle 2A \rangle$ and lies in a block with the same group as defect group.
\end{prop}

\begin{proof}
(a) The first two terms of the reduced Lefschetz module, the projective covers of $\varphi_{12}$ and
$\varphi_{13}$, lie in the principal block. Also $\varphi_{15}$ is projective and lies in block $2$, with defect
zero. Next, $\varphi_{16}$ is not projective, and lies in block $3$ with defect one.
This formula will be shown to be valid at the level of the Green ring
of virtual modules, and not just the Grothendieck ring of characters.

\smallskip
The alternating sum of the induced characters was computed using GAP\cite{gap}, and the character corresponding to the formula for $\widetilde{L}_{Fi_{22}}(\Delta)$ given above was obtained. The eight terms lying in $H_3 = 2^6:Sp_6(2)$ can be combined to yield the character of the inflation of the Steinberg module for the symplectic group.

\medskip
(b) Proposition $3.9$ (also see Remark $3.10$) tells us that the fixed point set $\Delta^{2B}$ is contractible.
The contractibility of $\Delta ^{2B}$ implies that $\Delta ^Q$ is mod $2$ acyclic for any $2$-group $Q$
containing an involution of type $2B$ (by Smith theory), and thus the reduced Lefschetz module
$\widetilde{L}_{N_G(Q)}(\Delta ^Q)=0$.
A theorem due to Burry and Carlson \cite[Theorem 5]{bc82} and to Puig (unpublished) was applied by Robinson
\cite[in Corollary 3.2]{rob88} to Lefschetz modules to obtain the following result; also see \cite[Lemma 1]{sa06}. The number of indecomposable summands
of $\widetilde{L}_G(\Delta)$ with vertex $Q$ is equal to the number of
indecomposable summands of $\widetilde{L}_{N_G(Q)}(\Delta ^Q)$ with the same vertex $Q$.
The proof of this result uses the Green correspondence, and the relationship to the Brauer
correspondence permits a conclusion regarding the blocks in which the summands lie.
This implies that the vertices of the indecomposable
summands of $\widetilde{L}_{Fi_{22}}(\Delta)$ do not contain any $2$-central involutions.

\smallskip
Proposition $4.6$ applies to the element $2C$ since $O_C = O_2(C_{Fi_{22}}(2C))$
contains $2$-central elements. The elementary abelian $2^5 \leq 2^{5+8}$ is of the type
$2^5 = 2A_6B_{15}C_{10}$; there is a purely $2$-central $2^4 \leq 2^5$.
Note that $\Delta^{2C}$ being contractible implies that no vertex
of a summand of $\widetilde{L}_{Fi_{22}}(\Delta)$ contains an involution of type $2C$.

\medskip
(c) Theorem $4.12$ applies to the fixed point set $\Delta^{2A}$, since the quotient group of the centralizer $C_{Fi_{22}}(2A)/O_2(C_{Fi_{22}}(2A))=U_6(2)$ is a Lie group. Therefore $\Delta^{2A}$ is homotopy
equivalent to the building for $U_6(2)$. The reduced Lefschetz module associated to this action
on the building is the irreducible Steinberg module. The vertex under the action of $2.U_6(2)$
is $\langle 2A \rangle$.  Thus there exists one indecomposable summand of the
Lefschetz module with vertex $2 = \langle 2A \rangle$, lying in a block with the same group
as defect group (block $3$ with defect one).  Note that since this block has cyclic defect
group of order two, there is only one nonprojective indecomposable module lying in this
block, namely the irreducible module $\varphi_{16}$.
Also note that since $\Delta^{2A}$ is homotopy equivalent to a building, $\Delta^Q$ will be contractible for any $2$-group $Q$ of order at least four which contains an involution of type $2A$. This implies that such a group $Q$ cannot be a vertex of a summand of $\widetilde{L}_{Fi_{22}}(\Delta)$.

\medskip
\item[(d)] It follows from the previous steps that there is precisely one nonprojective summand of the reduced
Lefschetz module; it has vertex $\langle 2A \rangle$ and it lies in a block with the same group as defect group.
The remaining summands are projective and are determined by their characters, and so
the formula for $\widetilde{L}_{Fi_{22}}(\Delta)$ is valid at the level of the Green ring
of virtual modules, and not just the Grothendieck ring of characters.
\end{proof}

\bigskip
{\it The Conway group $Co_3$ and $p=3$}\\
The group $Co_3$ has parabolic characteristic $3$, and it has three conjugacy
classes of elements of order three, denoted $3A$, $3B$ and $3C$, with the $3$-central elements
being those of type $3A$. The normalizers are $N_{Co_3}(\langle 3A \rangle)=3^{1+4}_+:4S_6$,
$N_{Co_3}(\langle 3B \rangle)=3^5(2 \times S_5)$ and $N_{Co_3}(\langle 3C \rangle)=S_3 \times
L_2(8):3$. Note that $L_2(8):3 = Ree(3) = \; ^2G_2(3)$ is a Chevalley group having local characteristic $3$.

\begin{prop} Let $\Delta$ be the subgroup complex associated to the distinguished $3$-radical collection $\widehat{\mathcal{B}}_3(Co_3)$.
\item[(a)] The fixed point sets $\Delta^{3A}$ and $\Delta^{3B}$ are contractible, and $\Delta^{3C}$
is equivariantly homotopy equivalent to the building for $^2G_2(3)$.
\item[(b)] The reduced Lefschetz module $\widetilde{L}_{Co_3}(\Delta)$ has precisely one
nonprojective irreducible summand, which has vertex $\langle 3C \rangle$ and lies in a block
with the same group as defect group.
\end{prop}

\begin{proof}
Proposition $3.9$ (and Remark $3.10$) implies that the fixed point set $\Delta^{3A}$ of the
$3$-central element is contractible. Proposition $4.6$ implies that $\Delta^{3B}$ is contractible since $O_C=O_3(C_{Co_3}(3B))=3^5=3A_{55}B_{66}$ contains $3$-central elements.
The reduced Lefschetz module $\widetilde{L}_{Co_3}(\Delta)$ has no summands with a vertex containing either an element
of type $3A$ or an element of type $3B$, by a similar argument to the one in part (b) of the previous proposition.

\medskip
Theorem $4.12$ applies to $\Delta^{3C}$. The reduced Lefschetz module for the $^2G_2(3)$ building is the irreducible Steinberg module, and under the action of $N_{Co_3}(\langle 3C \rangle)$, the vertex is $\langle 3C \rangle$. Also,
since $\Delta^{3C}$ is equivariantly homotopy equivalent to a building, $\Delta^Q$ will be contractible
for any $3$-group $Q$ of order at least nine which contains an element of type $3C$.
This implies that such a group $Q$ cannot be a vertex of a summand of $\widetilde{L}_{Co_3}(\Delta)$. The aforementioned Burry-Carlson theorem in Robinson's formulation implies now the result of part (b).
\end{proof}

\bigskip
{\it The Harada-Norton group $HN$ and $p=5$}\\
The group $HN$ has parabolic characteristic $5$ and has five conjugacy classes of elements of order $5$, denoted $5A$, $5B$, $5C$, $5D$ and $5E$; see \cite{atlas}. The elements of type $5B$ are $5$-central. An element of type $5D$ is the square
of an element of type $5C$, so they generate a group of order $5$ denoted $5CD$. The centralizers
are $C_{HN}(5A)=5 \times U_3(5)$, $C_{HN}(5B)=5^{1+4}.2^{1+4}.5$, $C_{HN}(5CD)=5^3.SL_2(5)$ and
$C_{HN}(5E)=5 \times 5^{1+2}:2^2$.

\begin{prop} Let $\Delta$ be the subgroup complex associated to the distinguished $5$-radical collection $\widehat{\mathcal{B}}_5(HN)$.
\item[(a)]  The fixed point sets $\Delta^{5B}$, $\Delta^{5CD}$ and $\Delta^{5E}$ are contractible,
and $\Delta^{5A}$ is equivariantly homotopy equivalent to the building for $U_3(5)$.
\item[(b)]  The reduced Lefschetz module $\widetilde{L}_{HN}(\Delta)$ has precisely one
nonprojective summand, which has vertex $\langle 5A \rangle$ and lies in a block with the same
group as defect group.
\end{prop}

\begin{proof}
Proposition $3.9$ (and Remark $3.10$) applies to the $5$-central elements of type $5B$, so that the fixed point set $\Delta^{5B}$ is contractible. The two groups $5^3 = B_6(CD)_{25}$ and $5 \times 5^{1+2}$ contain $5$-central elements
(note that $Z(5^{1+2})=\langle 5B \rangle$), thus according to Proposition $4.6$ the fixed point sets $\Delta^{5CD}$ and $\Delta^{5E}$ are both contractible. Thus the reduced Lefschetz module $\widetilde{L}_{HN}(\Delta)$ has no summands with a vertex
containing elements of type $5B$, $5C$, $5D$ or $5E$.

\medskip
Theorem $4.12$ applies to the elements of type $5A$, so that $\Delta^{5A}$ is equivariantly homotopy equivalent to the building for $U_3(5)$. The reduced Lefschetz
module for this building is the irreducible Steinberg module, and under the action of $N_{HN}(\langle
5A \rangle)$, the vertex is $\langle 5A \rangle$. Thus $\widetilde{L}_{HN}(\Delta)$ has one
summand with vertex $\langle 5A \rangle$ lying in a block with the same group as defect group.
Since $\Delta^{5A}$ is homotopy equivalent to a building, $\Delta^Q$ will be contractible
for any $5$-group $Q$ of order at least $25$ which contains an element of type $5A$.
This implies that such a group $Q$ cannot be a vertex of a summand of $\widetilde{L}_{HN}(\Delta)$.
\end{proof}


\begin{thebibliography}{99}

\bibitem{bs04}D.J. Benson, S.D. Smith, {\it Classifying spaces of sporadic groups, and their 2-completed homotopy decompositions}, Mathematical Surveys and Monographs, AMS, Providence, RI (in press).

\bibitem{bc82}D.W. Burry, J.F. Carlson, {\it Restrictions of modules to local subgroups}, Proc. Amer. Math. Soc. 84 (1982) 181-184.

\bibitem{atlas}J.H. Conway, R.T. Curtis, S.P. Norton, R.A. Parker, R.A. Wilson, {\em Atlas of finite groups} (Oxford University Press, 1985).

\bibitem{gap}GAP, {\em Groups, algorithms and programming}, http://www-gap.idcs.st-and.ac.uk/~gap.

\bibitem{moc}MOC, {\em The Modular Atlas homepage}, http://www.math.rwth-aachen.de/~MOC/.

\bibitem{gls2b}D. Gorenstein, R. Lyons, R. Solomon,
{\em The classification of the finite simple groups}, Number 2, Part I, Chapter G,
(Mathematical Surveys and Monographs, vol. 40, AMS, Providence, RI 1996).

\bibitem{gr02}J. Grodal, {\it Higher limits via subgroup complexes}, Ann. of Math.(2) 155 (2002) 405-457.

\bibitem{gs}J. Grodal, S.D. Smith, {\it Propagating sharp group homology decompositions}, Adv. Math. 200 (2006) 525-538.

\bibitem{mgo3}J.S. Maginnis, S.E. Onofrei, {\it New collections of $p$-subgroups and homology decompositions for classifying spaces of finite groups}, to appear Comm. Algebra.

\bibitem{qu78}D. Quillen, {\it Homotopy properties of the poset of nontrivial $p$-subgroups of a group}, Adv. in Math. 28 (1978) 101-128.

\bibitem{rob88}G.R. Robinson, {\it Some remarks on permutation modules}, J. Algebra 118 (1988) 46-62.

\bibitem{sa06}M. Sawabe, {\it On the reduced Lefschetz module and the centric $p$-radical subgroups. II}, J. London Math. Soc. 73 (2006) 126-140.

\bibitem{sol74}R. Solomon, {\it On defect groups and $p$-constraint}, J. Algebra 31 (1974) 557-561.

\bibitem{tw91}J. Th\'evenaz, P.J. Webb, {\it Homotopy equivalence of posets with a group action}, J. Combin. Theory Ser. A 56 (1991) 173-181.

\end{thebibliography}
\end{document}